\documentclass[11pt]{article}
\usepackage[a4paper, total={6in, 9in}]{geometry}
\usepackage{authblk}

\usepackage{graphics} 
\usepackage{epsfig} 
\usepackage{times} 
\usepackage{amsmath} 
\usepackage{amssymb}  
\usepackage{amsthm} 
\usepackage{subcaption}
\usepackage{float}
\usepackage{algorithm}
\usepackage{algpseudocode}
\usepackage{xspace}
\usepackage{verbatim}
\usepackage{multirow}
\usepackage{csquotes}
\usepackage{booktabs}
\usepackage{enumerate}  
\usepackage{cite}
\usepackage{subcaption}

\newtheorem{theorem}{Theorem}

\newtheorem{lemma}{Lemma}
\newtheorem{remark}{Remark}
\newtheorem{assume}{Assumption}

\title{\LARGE \bf
	On the Design and Analysis of Multivariable Extremum Seeking Control  using Fast Fourier Transform }

\author{Dinesh Krishnamoorthy 
}
\affil{Department of Chemical Engineering, Norwegian University of Science and Technology, 7491 Trondheim, Norway\\
		{\tt\small dinesh.krishnamoorthy@ntnu.no}}
\date{}

\begin{document}
	
	\maketitle

\begin{abstract}                
 This paper proposes a multivariable extremum seeking scheme using Fast Fourier Transform (FFT) for a network of subsystems working towards optimizing the sum of their local objectives, where  the overall objective is the only available measurement. Here, the different inputs are perturbed with different dither frequencies, and the power spectrum of the  overall output signal obtained using FFT  is used to estimate the steady-state cost gradient w.r.t. each input. The inputs for the subsystems are then updated using integral control in order to drive the respective gradients to zero. This paper provides analytical rules for designing the FFT-based gradient estimation algorithm and analyzes the stability properties  of the resulting extremum seeking scheme for the static map setting. The effectiveness of the proposed FFT-based multivariable extremum seeking scheme is demonstrated using two examples, namely, wind farm power optimization problem, and a heat exchanger network for industrial waste-to-heat recovery. 
\end{abstract}


\section{Introduction}

Extremum seeking control (ESC) has emerged as a popular real-time optimization (RTO) and adaptive control  tool to optimize the steady-state output of a plant, despite lack of a mathematical model.  
The most popular approach to estimating the gradient is the classical sinusoidal perturbation-based extremum seeking control that dates back to the work of Draper and Li \cite{draper1951FirstESC}, and later gained popularity with the work of Krstic and Wang \cite{krstic2000stability}. The classical extremum seeking control  is based on perturbing the input with a slow sinusoidal dither signal, such that the plant appears as a static map. A high-pass filter is used to remove the DC-component of the cost measurement, and the resulting detrended sinusoidal response of the output signal (with zero mean) is correlated with the perturbation signal. The product of the two sinusoids will have a DC component, which is extracted using a low-pass filter. This  is approximately the steady-state cost gradient as explained in \cite{krstic2000stability,tan2010ESC}. 
The classical extremum seeking control thus involves the use of a high-pass and low-pass filter such that the dither signal is in the pass band of the filters \cite{tan2010ESC}. 

Although extremum seeking control was historically developed for single-input single-output (SISO) systems,  there has been several developments in  extremum seeking control for  large-scale networks with many subsystems, with applications ranging from wind farm power optimization\cite{menon2014collaborative,ebegbulem2018power,ghaffari2014power,creaby2009maximizing}, maximum power point tracking for photovoltaic grids \cite{ghaffari2015multivariable,ghaffari2014power}, oil production network \cite{DK2019DYCOPS,pavlov2017CDC}, heat exchanger networks \cite{mu2017real,zitte2018extremum}, formation flight control \cite{binetti2003formation} etc. to name a few. 
Development of extremum seeking control for such multi-input systems can be broadly categorized into multivariable ESC and distributed ESC. 

Multivariable extremum seeking control considers a multi-input single-output (MISO) case, where the objective is to estimate the steady-state gradient of the overall cost w.r.t. each input  \cite{ghaffari2015multivariable,ghaffari2012multivariable}. 
Distributed extremum seeking on the other hand assumes that each subsystem in the network has access to its local cost measurement. Each subsystem then locally employs  single-input single-output (SISO) extremum seeking loop, 
and a consensus estimation algorithm is used for coordination among the different subsystems as explained in  \cite{dougherty2014extremum,guay2018distributed,menon2014collaborative}. 
This paper deals with the former class, where we assume that the different subsystems in a network work towards optimizing the sum of their local objectives, however the local cost of each subsystem is not available. 
Maximizing  total power generation from a wind  farm or photovoltaic grid, maximizing oil and gas production from a network of wells, maximizing heat recovery in a heat exchanger network, minimizing total energy consumption in a processing plant or in a commercial building are a few examples of such systems. 

As noted in \cite{krstic2000stability,tan2010ESC}, the choice of the filter cut-off frequencies in the classical extremum seeking approach depends on the input perturbation frequency.
With multiple inputs and a single output, tuning and designing the filter cut-off frequencies to retrieve the gradient w.r.t each input channel  from the single cost measurement may not be trivial in practice. This is due to the fact that the  output signal now contains the frequency components of the different  input perturbations. Moreover, using the same low-pass filter on each channel may also lead to additional harmonics in the gradient estimation as noted in \cite{ghaffari2015multivariable}. In short, good performance of the multivariable extremum seeking control  using this approach requires  good tuning of the filter cut-off frequencies, which can be challenging and time consuming, especially in large-scale networks with several inputs. 


A  natural and systematic approach to analyze the effect of different frequency components in a signal  is to use the Fourier transform for spectral analysis, which maps the time series data into a series of frequencies. 
 The underlying idea behind the proposed multivariable  FFT-based extremum seeking scheme is as follows: 
For a static map as shown in Fig.~\ref{Fig:Idea}, at $ u<u^* $ a sinusoidal perturbation in the input $ a\sin \omega t $, generates a sinusoidal variation in the output at the same frequency $ b \sin \omega t $. The power spectrum of the output signal would then have an amplitude of $ b $ at frequency $ \omega $. Similarly, for $ u>u^* $,  a sinusoidal perturbation in the input $ a\sin \omega t $, generates a sinusoidal variation in the output at the same frequency shifted by 180$ ^\circ $ resulting in $ b' \sin (\omega t+ \pi) $. Also in this case, the power spectrum of the output signal would  have an amplitude of $ b' $ at $ \omega $. For $ u \approx u^* $, a sinusoidal perturbation in the input $ a\sin \omega t $, generates a periodic signal with twice the frequency. Consequently, the power spectrum would indicate an amplitude of zero at $ \omega $, since the power of the signal at this frequency is zero.  Instead the power spectrum would have a non-zero amplitude  at $ 2\omega$. For example, if we have a  quadratic static map $ y = (u-u^*)^2 $, adding $ a\sin \omega t $ at $ u^* $ would result in an output variation of $y =a^2sin^2(\omega t)=0.5a^2(1 - \cos(2\omega t))   $, which clearly only has a frequency component at $ 2\omega $. The amplitude spectrum, which shows the strength of the variations of a signal at different frequencies is thus a very good  indicator of the gradients at the  different perturbation frequencies. Another advantage of using FFT-based extremum seeking is that it is robust to slowly varying unmeasured disturbances as long as they are not in the same frequency as any of the input perturbations.

 \begin{figure}
	\centering
	\includegraphics[width=0.5\linewidth]{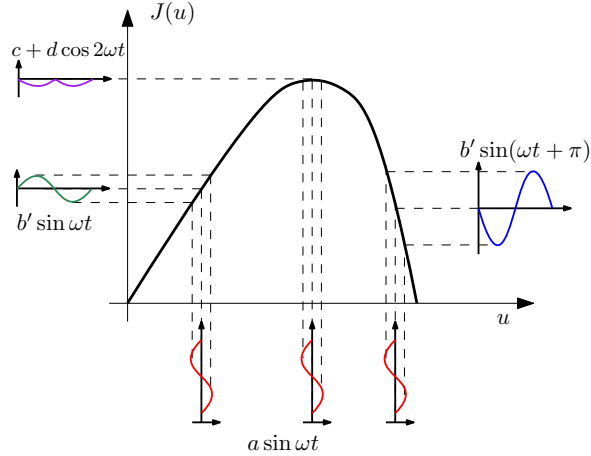}
	\caption{Effect of sinusoidal perturbation of the input-output static map at different points.}\label{Fig:Idea}
 \end{figure}
 In the context of extremum seeking control, Fourier transform has been predominantly used to extract the phase information in order to compensate for the phase shift introduced by the plant dynamics, such that the static map assumption can be relaxed \cite{munteanu2009wind,atta2015extremum,bizon2016global}.  For example, the authors in  \cite{atta2015extremum} estimate the different harmonics using a Kalman filter, and drive the high frequency harmonics to zero when the perturbation and the plant dynamics are in the same timescale.
On the other hand, in this paper we use FFT to directly extract the gradient information w.r.t multiple input channels from the amplitude spectrum using the two-timescale separation principle (just as in the case with classical extremum seeking control). The phase spectrum may be used only to estimate the sign of the gradient. 
The intuitive nature of FFT and the amplitude spectrum has previously led to the use of gradient estimation using FFT in a few engineering applications such as \cite{corti2014automatic,beaudoin2006bluff}. However, these works  provide little information into designing the FFT-based extremum seeking scheme and  analyze its  properties, which this paper aims to address.

To this end, this paper formalizes  the multivariable extremum seeking scheme using fast Fourier transform (FFT) and studies its stability properties.  In particular, the main contribution of this paper is to provide analytical  rules on the choice of the tuning parameters, namely, the window frame length $ N $, and the integral gain  $ K_{i} $ that ensures closed-loop stability of the proposed  FFT-based extremum seeking scheme.  The effectiveness of the FFT-based extremum seeking scheme is demonstrated using a wind farm optimization example, as well as on a industrial waste-to-heat recovery network. 
\section{Extremum seeking control using Fast Fourier Transform (FFT)}

Consider a network of $ n $ subsystems working towards  optimizing the sum of their local objectives 
The  inputs are denoted by $ \mathbf{u}= [u_{1},\dots,u_{n}]^{\mathsf{T}}\in \mathbb{R}^{n} $, and the overall  cost function $ J(\mathbf{u}) \in \mathbb{R}  $ is denoted by the static map $
J(\mathbf{u}) = \sum_{i=1}^{n} \xi_{i}({u}_{i})$
where $ \xi_{i}({u}_{i}) $ denotes the unknown functions representing the local cost for each subsystem $ i $. 
The objective is to minimize the overall cost of the network
\begin{equation}\label{Eq:Cost}
\min_{\mathbf{u}\in \mathbb{R}^{n} } \; J(\mathbf{u}) =  \sum_{i=1}^{n} \xi_{i}({u}_{i})
\end{equation}
\begin{assume}\label{asm:costmeas}
	The overall cost measurement $ J \in \mathbb{R}$ is additively separable, but the local cost measurements $ \xi_{i}({u}_{i}) $ are not available.
\end{assume}
The above assumption implies that we consider a multi-input single-output (MISO) plant. 
Without loss of generality, we state the following assumption for a minimization problem. 
\begin{assume}\label{asm:smoothness}
 The nonlinear map  $ J(\mathbf{u}) $ is sufficiently smooth and continuously differentiable with a unique minimum at $ J(\mathbf{u}^*) $, such that
 \begin{subequations}
  \begin{align}
	\frac{\partial J}{\partial {u}_{i}} ({u}_{i}^*)  & = 0 \quad \forall i \in \mathbb{I}_{1:n}\label{Eq:NCO} \\
	\frac{\partial^2 J}{\partial{u_{i}}^2}  ({u}_{i}^*) &>0 \quad \forall i \in \mathbb{I}_{1:n}
\end{align}
Furthermore $ \exists \;\alpha_2 \geq \alpha_1 >0 $ such that
\begin{align}
	\alpha_1 \tilde{ u_{i}}^2  &\leq \frac{\partial J}{\partial {u}_{i}}  \tilde u_{i} \leq 	\alpha_2 \tilde{ u_{i}}^2   \label{Eq:Attractiveness}
\end{align}
 \end{subequations}
where  $ \tilde{u}_{i} :={u}_{i} - {u}_{i}^* $
\end{assume}

In this paper, we study the multivariable extremum seeking control of such a process, without the need for a mathematical  model of the plant. The objective is to steer the system to the equilibrium $ \mathbf{u}^* $ using only the  overall cost measurement $ J $. 

\begin{figure*}[t]
	\begin{subfigure}{0.5\linewidth}
		\centering
		\includegraphics[width=0.9\linewidth]{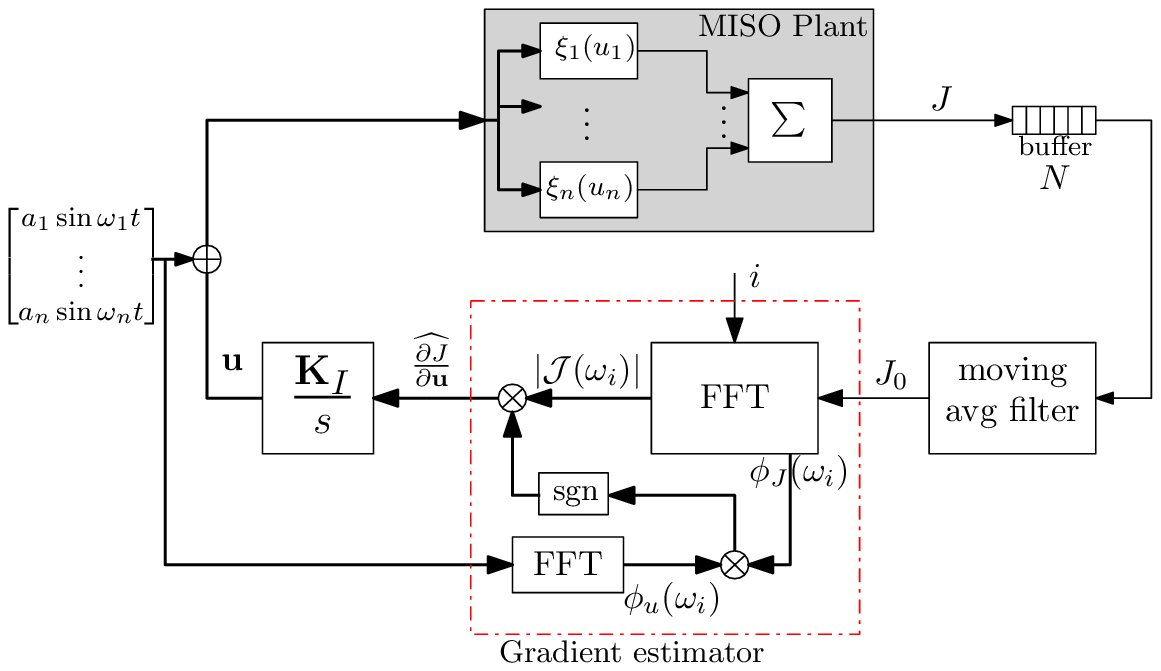}
		\caption{}\label{Fig:BlockDiagram}
	\end{subfigure}
	\begin{subfigure}{0.5\linewidth}
		\centering
		\includegraphics[width=0.9\linewidth]{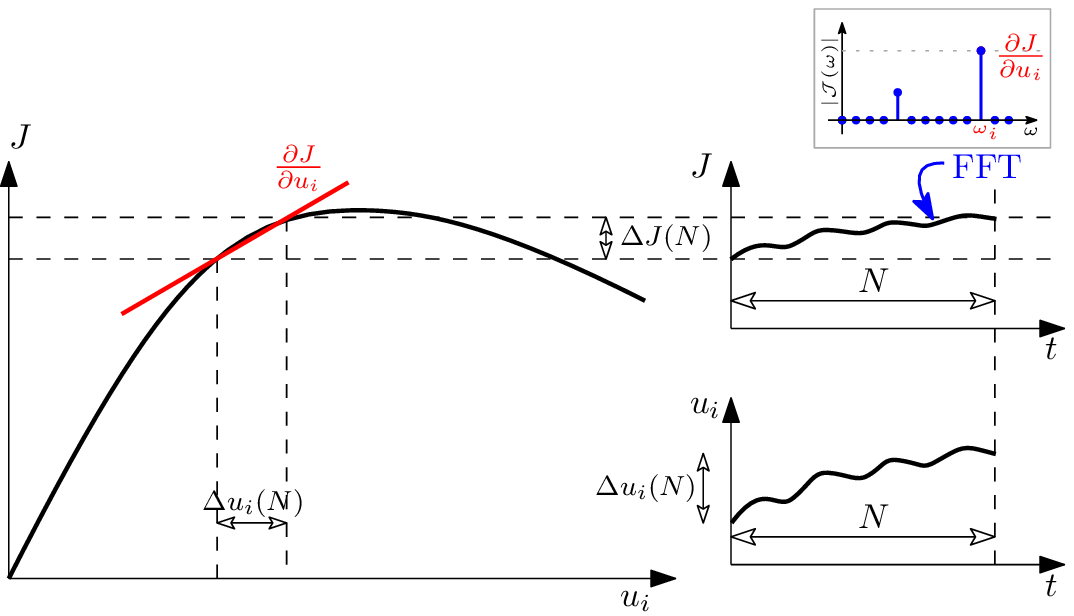}
		\caption{  }\label{Fig:LocalLinearGrad}
	\end{subfigure}
	\caption{ (a) Block diagram of the proposed FFT-based extremum seeking scheme. (b) Illustration of the local gradient estimation with respect to $ u_{i} $ using spectral analysis of the past $ N  $ samples of cost measurement. }
\end{figure*}

Each input channel is superimposed with a unique  sinusoidal dither signal  $ a_{i}\sin \omega_{i}t $, with amplitude $ a_{i} $ and frequency $ \omega_{i} $. 
The overall output cost measurement $ J $ is then composed  of the frequency components $ \omega_{i} $, for all $ i= 1,\dots,n $.  Using spectral analysis we can decompose the output signal $ J $ into different frequency components. 
To do this, we use a sliding window of fixed length with $ N $ past data points to get the $ N- $point DFT using the finite set of  cost measurement $ [J(0),\dots,J(N-1)]^{\mathsf{T}} $, and input signals $ [u_{i}(0),\dots,u_{i}(N-1)]^{\mathsf{T}} $ for all $ i  \in \mathbb{I}_{1:n} $. 
To perform FFT, the signals need to be detrended such that they have zero mean. Hence, the  DC-component from the time domain data  is removed, for example by using a moving average filter.  The detrended signals with zero-mean are denoted by the superscript $ ({\cdot})^0 $. 

\begin{remark}
The detrending step is analogous to the high pass filter in the classical extremum seeking control, where the  purpose of the high pass filter is also to remove the DC component from the cost measurement $ J $ \cite{krstic2000stability}. 
\end{remark}

Let the past $ N $ discrete samples of the detrended output cost measurement $ J^{0} $, be denoted by $ [J^{0}(0),J^{0}(1),\dots,J^{0}(N-1)]^{\mathsf{T}} $. The discrete-time Fourier transform (DTFT) reads as
	$\mathcal{J}(\omega) = \sum_{k =-\infty}^{\infty}J^{0}(k) e^{-j\omega\,k} $
and the $ N $-point DFT  samples $ 	\mathcal{J}(\omega) $ at $ N $ frequency components given by,
\begin{equation}\label{Eq:DFT}
\mathcal{J}(l) = \sum_{k = 0}^{N-1}J^{0}(k) e^{-j\frac{2\pi}{N}l\,k} \qquad \forall l = 0,\dots,N-1
\end{equation}

$ \mathcal{J}(l)  $ is a complex valued function, which has a corresponding magnitude spectrum denoted by $ |\mathcal{J}(l) |$, and a phase spectrum denoted by $ \phi_J(l) $.
The magnitude of the cost function for the $ n $ input dither frequencies $ \omega_{i} $  for all $ i \in \mathbb{I}_{1:n} $ can then be obtained from the single-sided amplitude spectrum $ 2|\mathcal{J}(l)| $ for all $ l = 1,\dots,N/2 $, where we extract the frequency components of interest from the $ N $-point DFT. That is, the magnitude of the cost gradient w.r.t. the $ i^{th} $ input channel is given by
\begin{equation}\label{Eq:Amplitude}
 \left|\frac{\partial J}{\partial u_{i}}\right|  =  \frac{2}{a_{i}}|\mathcal{J}({\omega_{i}}N/{2\pi})| \quad \forall \; i \in \mathbb{I}_{1:n}
\end{equation}
Hereafter, we denote $ |\mathcal{J}({\omega_{i}}N/{2\pi})| $ simply as $  |\mathcal{J}(\omega_{i})|  $ for the sake of notational simplicity. Since the amplitude spectrum $  |\mathcal{J}(\omega_{i})| >0$, the phase information $ \phi_{J}(\omega_{i}) $ with respect to the phase of the input signal $ \phi_{u_{i}}(\omega_{i}) $ is used to determine the sign of the gradient. 
The phase $ \phi_{J}(\omega_{i}) $ is obtained from the FFT, and similarly, the phase of the input perturbations $ \phi_{u_{i}}(\omega_{i}) $ can be obtained using FFT of the dither signals. 

Together, the gradient of the cost w.r.t to the $ i^{th} $ input is then estimated as follows:
\begin{equation}\label{Eq:FFTgrad}
\widehat{\frac{\partial {J}}{\partial u_{i}} } =  \frac{2}{a_{i}}\left|\mathcal{J}(\omega_{i})\right| \textup{sgn}\left[\frac{\phi_{J}(\omega_{i})}{ \phi_{u_{i}}(\omega_{i}) }  \right] , \; \forall i \in \mathbb{I}_{1:n}
\end{equation}
This is schematically shown in Fig.~\ref{Fig:BlockDiagram} for the $ i^{th} $ input, where spectral analysis of the cost measurement from the past $ N $ samples is used to estimate the local  gradient $ \widehat{\frac{\partial {J}}{\partial u_{i}} }  $. 
Alternatively, the estimated gradient can equivalently be computed directly as
\begin{equation}
\widehat{\frac{\partial {J}}{\partial u_{i}} }   =  \frac{\Re(\mathcal{J}(\omega_{i})) }{\Re(\mathcal{U}_{i}(\omega_{i})) }, \; \forall i\in \mathbb{I}_{1:n}
\end{equation}
where $ \Re(\cdot) $ denotes the real part of a complex number, and $ \mathcal{U}_{i}(\omega) $ is the DFT of the $ i^{th} $ input signal. 
\begin{remark}
	The Euler's formula establishes the fundamental relationship between the complex exponential function and the trigonometric functions in \eqref{Eq:DFT}.  Thus, the product of the detrended cost $ J_{0} $ and exponential terms  in \eqref{Eq:DFT} can be seen as  analogous to the product of the detrended output signal and the sinusoidal dither in classical extremum seeking control, which provides the demodulation effect at the perturbation frequency. 
\end{remark}

Once the gradient is estimated, the inputs can be  updated using a gradient descent step, which is simply an  integral  controller to drive the estimated gradients to zero,
\begin{equation}\label{Eq:ControlLaw}
{\dot{u}_{i}} = - {K}_{i} \widehat{ \frac{\partial {J}}{\partial{u}_{i}}}, \forall i \in \mathbb{I}_{1:n}
\end{equation}
where $ K_{i} $ is the integral gain  for the $ i^{th} $ input.  
The perturbation signals are added to the nominal input $ {u}_{i} $ computed by the integral controller, and the resulting input  is applied on the plant. 
The sketch of the proposed FFT-based extremum seeking control is shown in Algorithm~\ref{FFTESC:alg}.

\begin{algorithm}[t]
	\caption{FFT-based Extremum seeking control}
	\label{FFTESC:alg}
	\begin{algorithmic}[1]
		\Require $ N $,  $ \omega_{1},\dots,\omega_{n} $,  $ a_{1},\dots,a_{n} $  and $ {K}_{i} $
		\State	$ J^{0} $ $\leftarrow$ Detrend the output cost measurement $ J $
		\State $ \mathcal{J}({l}) = \sum_{k = 0}^{N-1}J^{0}(k) e^{-j\frac{2\pi}{N}l\,k}  $  \Comment perform FFT
		\For { $ i = 1,\dots,n $ (in parallel)}
		\State $ \mathcal{U}_{i}({l}) = \sum_{k = 0}^{N-1}u^{0}_{i}(k) e^{-j\frac{2\pi}{N}l\,k}  $  \Comment perform FFT
		\State $\left| \mathcal{J}(\omega_{i})\right|,\phi_{J}(\omega_{i})   \leftarrow$ amplitude  and phase of the cost  at the $ i^{th} $ input frequency
		\State $\left| \mathcal{U}_{i}(\omega_{i})\right|,\phi_{u_{i}}(\omega_{i})   \leftarrow$ amplitude  and phase of the input at the $ i^{th} $ input frequency
		\State $ \widehat{\frac{\partial J}{\partial u_{i}}}  \leftarrow  \frac{2}{a_{i}}\left|\mathcal{J}(\omega_{i})\right| \textup{sgn}\left[\frac{\phi_{J}(\omega_{i})}{ \phi_{u_{i}}(\omega_{i}) }  \right]  $
		\State $ {\dot{u}_{i}} = - {K}_{i} \widehat{ \frac{\partial {J}}{\partial{u}_{i}}} $
		\EndFor
		\Ensure $\mathbf{u} $ 
	\end{algorithmic}
\end{algorithm}

\section{Stability Analysis}

\begin{assume}[{Choice of dither signal}]\label{asm:uniqueDither}
	\begin{enumerate}[(i)]
		\item The output signal is persistently exciting of at least order $ 2n $. 
		\item  $ \max(\{\omega_i\}) $ is chosen such that the  plant can be assumed to be a static map.
		\end{enumerate}
\end{assume}
Assumption~\ref{asm:uniqueDither}(i) can be satisfied if the perturbation frequencies of the different input channels are mutually independent, i.e. $ \omega_{i} \neq \omega_{j} $, $ 2\omega_{i} \neq \omega_{j} $,  $ \omega_{i}+\omega_{j}\neq \omega_{k} $ for any distinct $ i $, $ j  $, and $ k $.
This means that the frequency of one input perturbation must not be the same or lie in the harmonics of perturbation of the other input channels. This is to ensure that unique persistence of excitation condition is satisfied \cite{ghaffari2012multivariable}. 
With Assumption~\ref{asm:uniqueDither}(ii) we restrict our analysis in this section to the static map setting, where the dynamics of the extremum seeking scheme are governed by \eqref{Eq:ControlLaw}.

When we use an $ N $-point DFT to obtain the amplitude spectrum, it is based on spectral analysis on a finite set of data with length $ N $. The choice of the window length is   an important factor which affects the accuracy of the gradient estimation. If $ N $ is not carefully chosen, then it is possible that the $ N $-point DFT may not sample the DTFT precisely at one or more of the perturbation frequencies. That is,  the discrete frequency array $ l = \{0,\dots,N-1\}$ may not contain one or more of the perturbation frequencies exactly (cf. \eqref{Eq:Amplitude}). This leads to \textit{spectral leakage} resulting in bias in the spectral values, and consequently bias in the estimated gradient.   The choice of the window length and its effect on the gradient estimation is formalized in the following theorem. 
\begin{theorem}[Minimum window length $ N $]\label{thm:sizeN}
	Given assumption~\ref{asm:uniqueDither}, 
 the length of the moving window frame $ N $ that ensures  no spectral leakage in the N-point DFT of  the output signal $ J^0 $ 
  is given by 
	\begin{equation}\label{Eq:LCM}
	N = \gamma \;\textup{lcm}\left(\left\{\frac{2\pi}{\omega_{i}}\right\}_{i=1}^{n} \right)
	\end{equation}
for some $ \gamma \in \mathbb{Z}_{+} $ and where $ \textup{lcm}(\cdot) $ denotes the least common multiple.
\end{theorem}
\begin{proof}
We can consider the finite segment of a signal $ x_{N}(k) $ as a product of an infinite time series $ x(k) $ and a rectangular data window \[ w(k): = \begin{cases}
	1 & 0\le k \le N-1\\
	0 & \text{elsewhere}
\end{cases}  \]
The Fourier transform of $ x_{N}(k) $ is then given by 
\begin{equation*}
	\mathcal{X}_{N}(\omega) = \mathcal{X}(\omega) \otimes \mathcal{W}(\omega) 
\end{equation*}
where $ \otimes $ denotes convolution. For a rectangular window, we have
\[ \mathcal{W}(\omega) = \left[ \frac{\sin(\omega N/2)}{\sin(\omega/2)} \right] e^{-j\omega(N-1)/2} \] 
where the periodic sinc function has mainlobes and sidelobes in the DTFT. 
Sampling this at $ \omega = 2\pi l/N $ for $ l = 0,\dots,N-1 $ one can see that
\[ \mathcal{W}(l) = \begin{cases}
	N & l = 0\\
	0 & l = 1,\dots,N-1
\end{cases} \]
which implies that the sidelobes are sampled at their zero values. Convolution with the sinusoidal signal $ a_{i}\sin\omega_{i}t $ gives a frequency-shifted periodic sinc centered around $ \pm \omega_i $ scaled by the amplitude $ a_{i} $.  By choosing $ N $ as \eqref{Eq:LCM} we ensure $ \omega_i \in \left\{\frac{2\pi l}{N} \right\}_{l=0}^{N-1}$, and the integer number of cycles for all the frequency components within the $ N $ samples ensures that there is no spectral  leakage to nearby DFT bins \cite{mulgrew1999digital,manolakis2000statistical}. 	
\end{proof}
An alternative  proof of  Theorem~\ref{thm:sizeN} can be given by looking at the
 $ N $-point DFT in \eqref{Eq:DFT}, where it can be seen that this is periodic with period $ N $. The inverse DFT of \eqref{Eq:DFT} expressed as 
\begin{equation}\label{Eq:invDFT}
	\hat{J}^{0}(k) = \frac{1}{N}\sum_{l=0}^{N-1} \mathcal{J}(l) e^{j\frac{2\pi}{N}lk}
\end{equation}
tells us that the reconstructed time series $ \hat{J}^{0}(k) $ is also periodic with period $ N $.  This is because sampling in the time domain results in periodicity in the frequency domain, and vice versa as given by the DFT sampling theorem. This implies that the two end points of the time series samples, $ J^{0}(0) $ and $ J^{0}(N) $, are interpreted as though they were connected together \cite[Ch. 5]{manolakis2000statistical}. Therefore, by choosing $ N $ as the least common multiple of the perturbation time periods \eqref{Eq:LCM}, we can reconstruct $ \hat{J}^{0}(k)  = J^{0}(k) $ that contains integer number of full time periods of  all the perturbation frequencies $ \left\{\frac{2\pi}{\omega_{i}}\right\}_{i=1}^{n} $. 
If $ N $ is not chosen according to \eqref{Eq:LCM}, then the reconstructed $ \hat{J}^{0}(k) $ contains discontinuities, which results in higher order harmonics in the power spectrum. In this case, the DFT samples the sidelobes of the DTFT at non-zero values leading  to  \textit{spectral leakage}. Furthermore, spectral leakage due to signals with non-integer number of cycles can also spread to other frequencies and can mask other signals, especially if their amplitude is  relatively small, causing gradient estimation errors w.r.t other inputs as well.
Therefore to avoid such issues,  the minimum window length $ N_{min} $ is given by \eqref{Eq:LCM} with $ \gamma = 1 $ in the FFT-based gradient estimation scheme. 

\begin{remark}[Minimum resolution]
In the multivariable setting, when we have several frequency components, the difference between any two perturbation  frequencies $ \omega_i $ and $ \omega_j $ must satisfy
\begin{equation}\label{thm:resolution}
	\left|\omega_i - \omega_j\right| > \frac{2\pi}{N-1}
\end{equation}
in order for it to be distinguishable in the amplitude  spectrum. This is due to the windowing effect, where the DTFT has mainlobes which has a width of approximately $ 2\pi /(N-1) $ \cite{manolakis2000statistical}. 
\end{remark}

\begin{remark}
	The proposed approach can be seen as a variant of the Gabor transform. That is, the use of a moving window to estimate the time-varying local gradient using FFT is similar to using a Gabor transform \cite{gabor1946theory,sejdic2009time,rotstein1999gabor}, where instead of a moving Gaussian window, we use a  moving rectangular window. 
\end{remark}

Due to the local linear approximation within the window,  using the past $ N $ data samples  obtained at time $ t $, may lead to some deviations in the estimated gradient $ \widehat{\frac{\partial {J}}{\partial u_{i}}}  $ from  the true gradient $ \frac{\partial {J}}{\partial u_{i}}(u_{i}(t))   $ as the nominal input $  u_{i}(t) $ changes. See the simple illustrative example in Appendix~\ref{app:example} that  clearly demonstrates this. The gradient estimation error due to the moving window is formalized in the following Lemma.  

\begin{lemma}[Boundedness of the gradient estimation error \cite{hunnekens2014dither}]\label{thm:EstErr}
 Consider the FFT-based gradient estimation scheme given by \eqref{Eq:DFT}-\eqref{Eq:FFTgrad}, where the gradient  estimated using the past $ N $ data samples is denoted as $ \widehat{\frac{\partial {J}}{\partial{u}_i} }$. Let the true steady-state cost gradient w.r.t the inputs $  u_{i}(t) $ be denoted as $ \frac{\partial {J}}{\partial{u}_i} ( u_{i}(t))  $. Given Assumptions~\ref{asm:costmeas}, \ref{asm:smoothness}, and \ref{asm:uniqueDither}, and $ N $ chosen according to Theorem~\ref{thm:sizeN}, the gradient estimation error  \begin{equation}\label{Eq:err}
\epsilon_{i}(t) :=\widehat{ \frac{\partial {J}}{\partial{u}_{i}}}- \frac{\partial {J}}{\partial {u}_{i}}({u}_{i}(t))
\end{equation}
is  bounded by,
\begin{align}\label{Eq:errBound}
|\epsilon_{i}(t)| \leq  K_iN &\max_{\tau \in [0,N]} \left| \mathbf{H}_i({u}_i(t-\tau)) \right| \max_{\tau \in [0,2N]} \left| \frac{\partial {J}}{\partial{u}_i}({u}_i(t-\tau)) \right|  
\end{align}
where $ \mathbf{H}_{i}(\cdot) := \frac{\partial^2 J}{\partial u_{i}}(\cdot)$ is the Hessian of the static map. 
\end{lemma}
\begin{proof}
	Since Assumption~\ref{asm:smoothness} holds, we can use the mean-value theorem to estimate the absolute error, where for some $  u_{i} $ in the interval $ \Delta  u_{i}(t) := \max_{\tau \in [0,N] } {u}_i(t-\tau) - \min_{\tau \in [0,N] } {u}_i(t-\tau)$, we have
	\begin{equation}\label{Eq:MVT}
	 |\epsilon_{i}(t)| \leq  \max_{\tau \in [0,N]} \left| \mathbf{H}_i({u}_i(t-\tau)) \right| \Delta  u_{i}(t) 
	\end{equation} 
	 Furthermore, since $  u_{i} $ is given by the integral controller $ \eqref{Eq:ControlLaw} $ and the estimated gradient is equal to the true gradient somewhere in the past $ N $ samples, 
	\begin{equation}\label{Eq:errBound2}
	\Delta {u}_{i}(t) \leq K_iN \max_{\tau \in [0,2N]} \left| \frac{\partial {J}}{\partial{u}_i}({u}_i(t-\tau)) \right| 
	\end{equation} 
	Combining \eqref{Eq:MVT} and \eqref{Eq:errBound2} proves the result. 
\end{proof}
It is also intuitively understood that the error bound is higher for large $ \Delta  u_{i} $ as shown in \eqref{Eq:MVT}.  From \eqref{Eq:err} and \eqref{Eq:errBound}  it can be seen that the gradient estimation error depends on the true gradient occurring somewhere in the last $ 2N $ samples. This can be seen as a time delay, and in order to study the stability properties with time delay, we  apply the Razumikhin condition  to the Lyapunov stability analysis \cite{gu2003stability}.
Given a positive definite function $ V(x) $, the Razumikhin condition states that a time delay system  with maximum time delay $ T $ is asymptotically stable if there exists a continuous non-decreasing function $ R(s) >s $, $ s>0 $ such that 
\begin{equation}\label{Eq:Razumikhin}
	\dot{V}(x(t))<0, \text{ whenever }R(V(x(t))) \geq \max_{\tau \in [0,T] } V(x(t-\tau))
\end{equation}

\begin{theorem}[Stability] \label{thm:convergence}
Consider the FFT-based extremum seeking scheme given by \eqref{Eq:DFT}-\eqref{Eq:ControlLaw}. Let Assumptions~\ref{asm:costmeas}, \ref{asm:smoothness} and \ref{asm:uniqueDither} hold, and $ N $ chosen according to Theorem~\ref{thm:sizeN}. 
Then for some $ d>1 $ there exists bounded  $ K_i  \in \left(0,\frac{\alpha_1}{\alpha_2 N\bar{\mathbf{H}}_{i} d}\right)$ for all $ i\in \mathbb{I}_{1:n} $   such that 
\[ \lim\limits_{t\rightarrow\infty}  u_{i}(t)  =  u_{i}^*  \]
\end{theorem}
\begin{proof}
	To analyze the stability of the optimum $  u_{i} =  u_{i}^* $, consider the candidate Lyapunov function,
	\begin{equation}\label{Eq:Lyapunov}
	V( u_{i}(t)) =  \frac{1}{2}\left(\tilde{ u_{i}}(t)\right)^2
	\end{equation}
	with $ \tilde{ u_{i}}(t) :=  u_{i}(t)- u_{i}^* $, and the derivative is written as,
		\begin{align*}
	\dot{V}( u_{i}(t)) &= \tilde{ u_{i}}(t)\dot{ u_{i}}(t) = -\tilde{ u_{i}}(t) K_i\widehat{\frac{\partial {J}}{\partial{u}_i} }  \\
	&= -\tilde{u_{i}}(t)K_i \left(\frac{\partial {J}}{\partial{u}_i} ( u_{i}(t) )+ \epsilon(t)\right) 
	\end{align*}
Upper bounding the Hessian by $ \bar{\mathbf{H}}_{i}:= \max\left| \mathbf{H}_i\right|  $ and  using Lemma~\ref{thm:EstErr}, we can write this as 
	\begin{align}
	\dot{V}( u_{i}(t)) &\leq -\tilde{u_{i}}(t)K_i \frac{\partial {J}}{\partial{u}_i} ( u_{i}(t) )
+ |\tilde{u_{i}}(t)|K_i^2N  \bar{\mathbf{H}}_{i} \max_{\tau \in [0,2N]} \left| \frac{\partial {J}}{\partial{u}_i}({u}_i(t-\tau)) \right| \nonumber\\
	&\leq -K_{i}\alpha_1 \tilde{u_{i}}^2 (t)+ |\tilde{u_{i}}(t)|K_i^2N  \bar{\mathbf{H}}_{i}\max_{\tau \in [0,2N]} \left| \alpha_2(\tilde{u}_i(t-\tau)) \right| \label{Eq:Lyap}
	\end{align}
where the last inequality comes from \eqref{Eq:Attractiveness} in Assumption~\ref{asm:smoothness}, which  implies   
$|\alpha_1\tilde{ u_{i}}|  \leq \left| \frac{\partial J}{\partial {u}_{i}} \right|  \leq 	 |	\alpha_2 \tilde{ u_{i}}| $.
Let $ R(s) := d^2s$ for some $ d>1 $, then using the Razumikhin condition (cf. \eqref{Eq:Razumikhin}), for asymptotic stability we require $ \dot{V}(u_{i})<0 $ if
\begin{align*}
	&\frac{1}{2} \max_{\tau \in [0,2N]}\tilde{ u_{i}}^2(t-\tau)  \leq \frac{1}{2} d^2 \tilde{ u_{i}}^2(t)  
	\Leftrightarrow \max_{\tau \in [0,2N]} |\alpha_{2}\tilde{ u_{i}}(t-\tau)|  \leq d \alpha_2| \tilde{ u_{i}}(t) |
\end{align*}
Using this in \eqref{Eq:Lyap}, we get
\begin{align*}
	\dot{V}( u_{i}(t)) &\leq -K_{i}\alpha_1 \tilde{u_{i}}^2 (t)+ K_i^2N  \bar{\mathbf{H}}_{i} d \alpha_2| \tilde{ u_{i}}(t) |^2 =\alpha_1 \tilde{u_{i}}^2(t) \left( -K_{i} + K_i^2N  \bar{\mathbf{H}}_{i} d \frac{\alpha_2}{\alpha_1}\right)
\end{align*}
For $ \dot{V}(u_{i})<0 $, we need $ \left( -K_{i} + K_i^2N  \bar{\mathbf{H}}_{i} d \frac{\alpha_2}{\alpha_1}\right)<0 $, and since this is a minimization problem, we need $ K_{i}>0 $ in \eqref{Eq:ControlLaw}. Therefore, $ \lim\limits_{t\rightarrow\infty}  u_{i}(t)  =  u_{i}^*  $, if for all $ i \in \mathbb{I}_{1:n} $
\begin{align*}
0<K_{i}<	\frac{\alpha_1}{\alpha_2 N\bar{\mathbf{H}}_{i} d}
\end{align*}

	
\end{proof}

To summarize, Theorem~\ref{thm:sizeN} provides a lower bound on the length of the window frame $ N $ as a function of the different input perturbation frequencies such that the gradient w.r.t all the inputs can be retrieved exactly from the amplitude spectrum of the output signal, and Theorem~\ref{thm:convergence} provides bounds on the controller gains $ K_{i} $ for each input channel such that the closed-loop system is stable under the static map condition.

\section{Case Example: Wind farm optimization}
In this section, we present a wind farm optimization example that demonstrate the performance of the proposed FFT-based gradient estimation scheme. The simulation results were performed in \texttt{MATLAB v2020b}, and the code can be found in the GitHub repository \cite{Github_ESCFFT}. 
In this example, we consider a wind farm with $ n = 6 $ identical wind turbines as shown in Fig.~\ref{Fig:WindFarm}.  
We assume the wind direction is along the $ x $ co-ordinate and orthogonal to the $ y $ co-ordinate, as shown in Fig.~\ref{Fig:WindFarm}. The wind turbines are uniform with diameter $ D=80m $ and blade roughness coefficient $ k= 0.075 $. The wind is in the direction of the positive $ x $-axis with a steady wind velocity of $ V_{\infty} = 8m/s$. The reader is referred to \cite{marden2013model} for the wind farm model used in the simulation.

The wake interaction within a wind farm can significantly affect the overall power generation. The objective is thus to compute the optimal  axial induction factor for the wind turbines $ \mathbf{u} = [u_{1},\dots,u_{n}] $, such that the total power generation from the wind farm $P(\mathbf{u})$ is maximized. We assume that the total  power  produced by the wind farm $ P \in \mathbb{R} $ is the only measurement available. 
\begin{figure}
	\centering
	\includegraphics[width=0.49\linewidth]{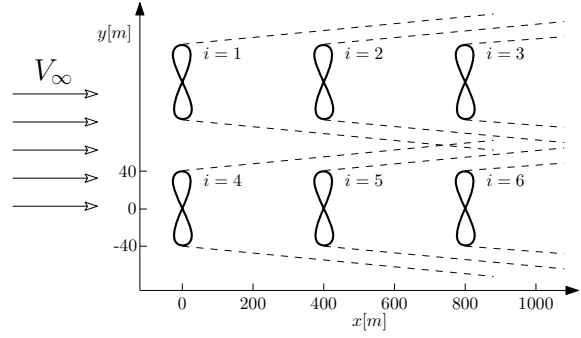}
	\caption{Wind farm layout with $ n=6 $ turbines with diameter 80 m, placed at (0,200), (400,200), (800,200)  and (0,0), (400,0), (800,0). }\label{Fig:WindFarm}
\end{figure}
\begin{figure*}
	\centering
	\begin{subfigure}{0.49\textwidth}
		\centering
		\includegraphics[width=\linewidth]{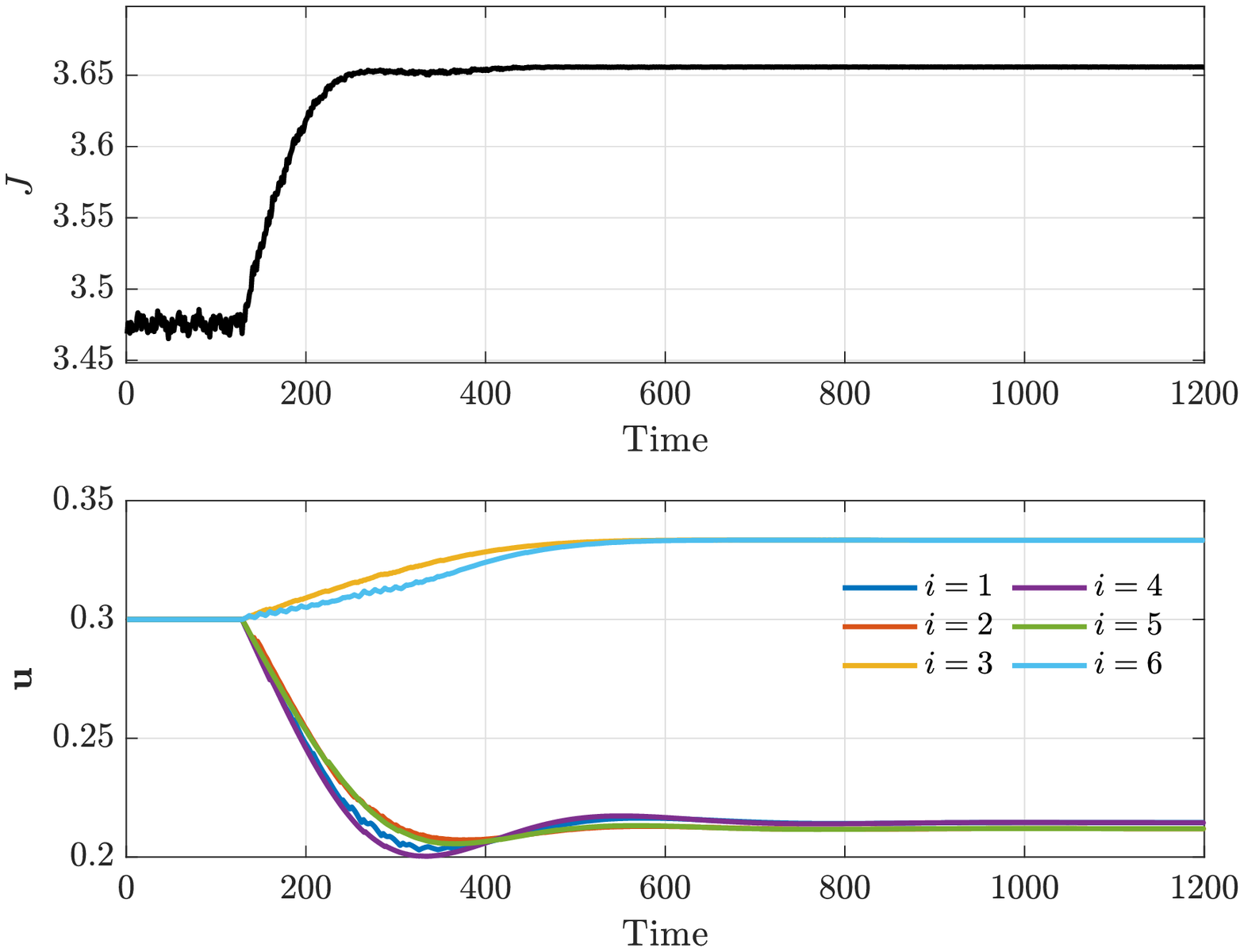}
		\caption{}\label{Fig:WFsim}
	\end{subfigure}
	\begin{subfigure}{0.49\textwidth}
		\centering
		\includegraphics[width=\linewidth]{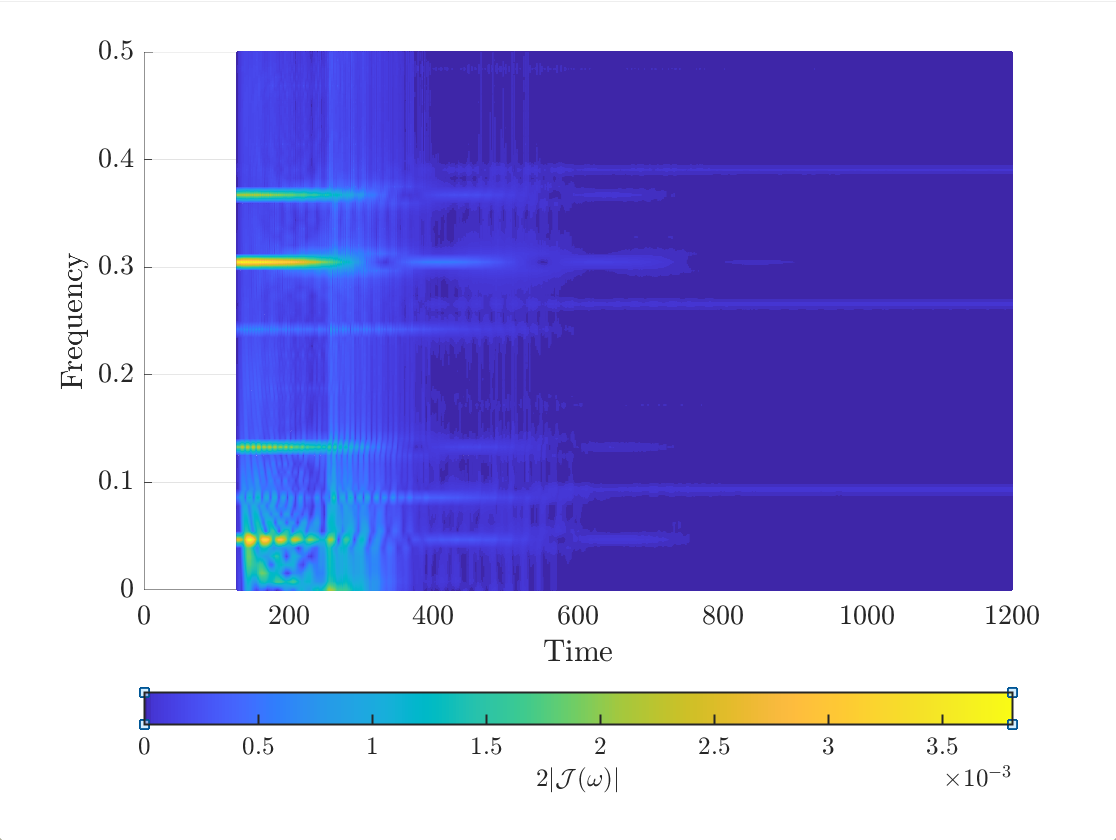}
		\caption{}\label{Fig:WFspectrogram}
	\end{subfigure}
	\caption{Example 1 without noise: (a) Simulation results showing the overall power production from the wind farm (in MW), the optimal axial induction factors for the three turbines, and the corresponding steady-state gradients estimated using the proposed method. (b) Spectrogram showing the power spectrum as a function of time. }
\end{figure*}
\begin{figure*}
	\centering
	\begin{subfigure}{0.49\textwidth}
		\centering
		\includegraphics[width=\linewidth]{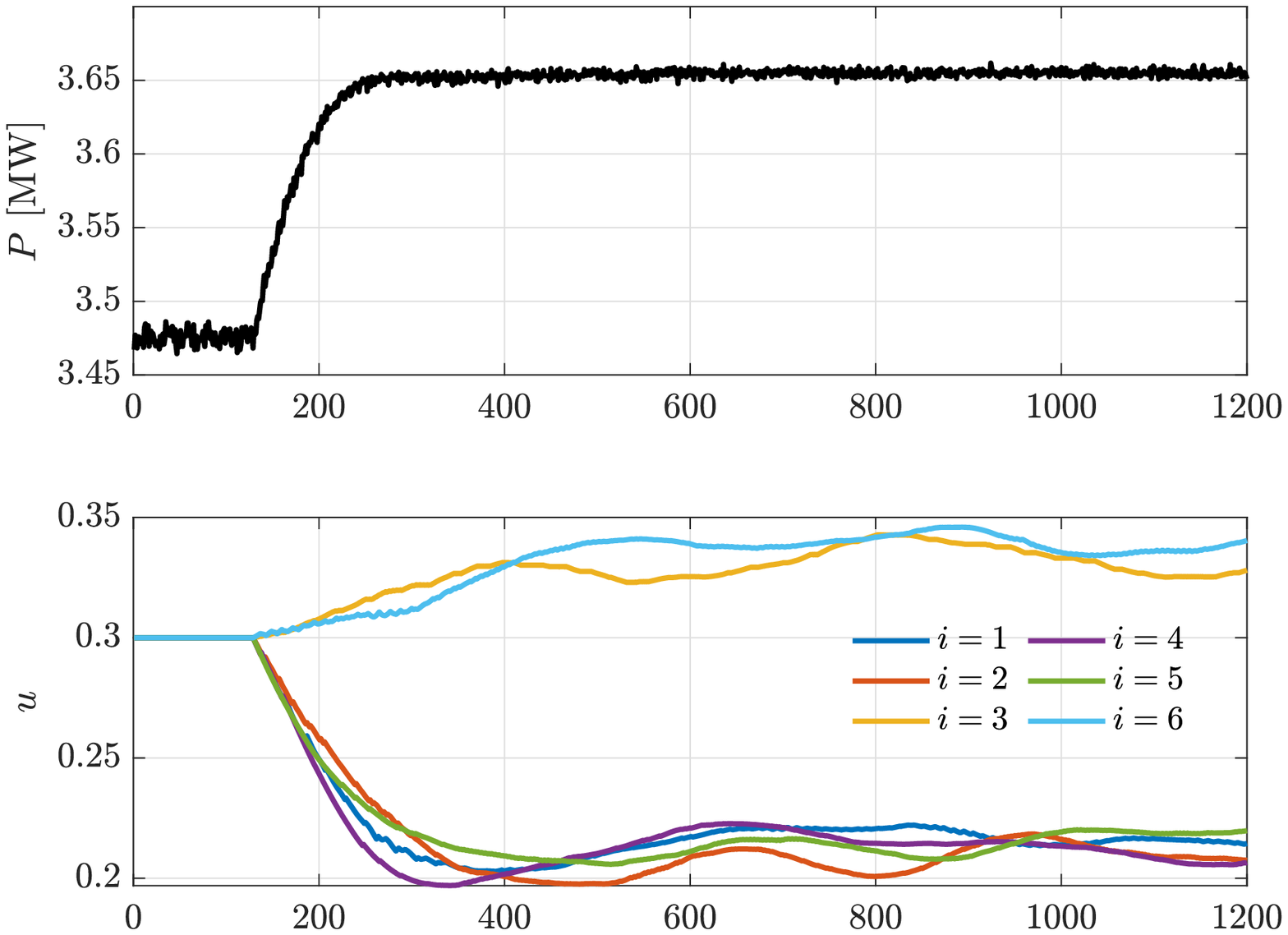}
		\caption{}
	\end{subfigure}
	\begin{subfigure}{0.49\textwidth}
		\centering
		\includegraphics[width=\linewidth]{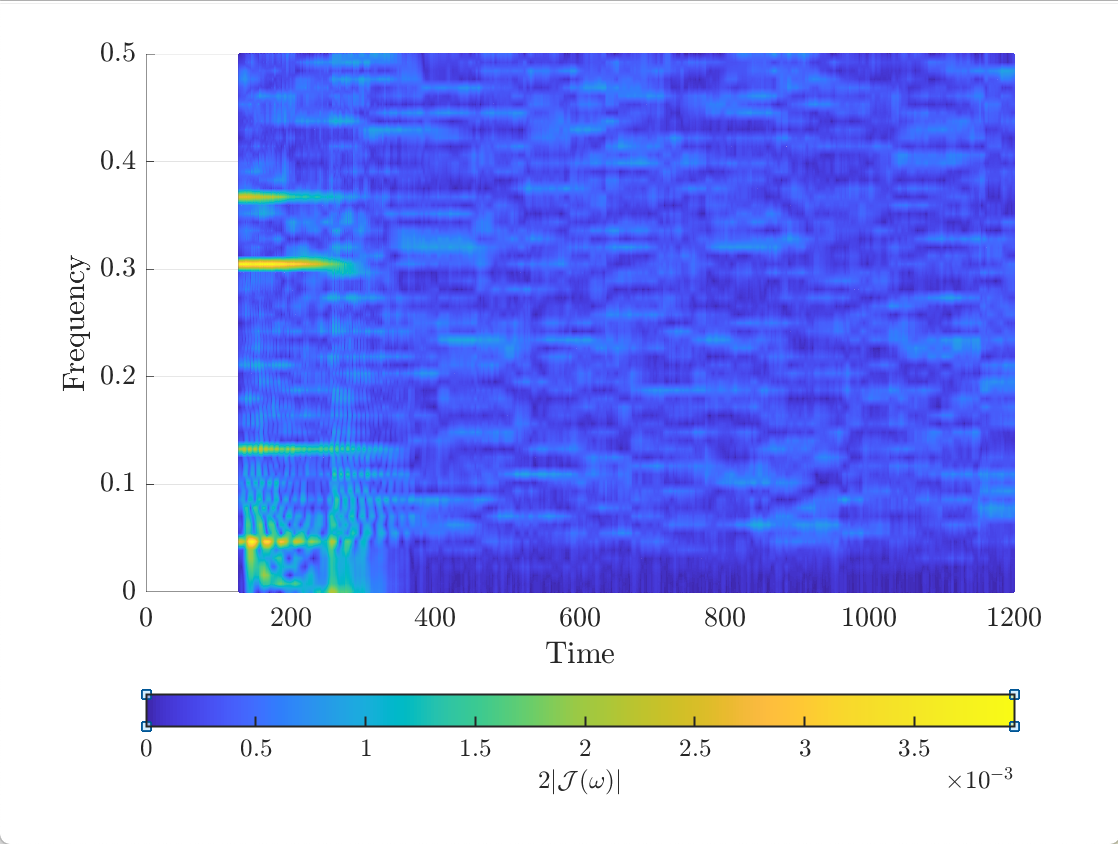}
		\caption{}
	\end{subfigure}
	\caption{Example 1 with noise: (a) Simulation results showing the overall power production from the wind farm (in MW), the optimal axial induction factors for the three turbines, and the corresponding steady-state gradients estimated using the proposed method. (b) Spectrogram showing the power spectrum as a function of time. } \label{Fig:WFsim2}
\end{figure*}

We set the initial axial induction factors to $ u_{i} =0.3$ for all the wind turbines. 
The six inputs channels are superimposed with sinusoidal signals with frequencies $ f = \begin{bmatrix}
	6&17&31&39&47&11
\end{bmatrix}/128  $ 
and amplitude $ a_{i}= $ 0.003.
The window length was chosen to be of $ N= $128 samples, according to Theorem~\ref{thm:sizeN}. 

The simulation results with the proposed FFT-based multivariable extremum seeking control is show in Fig.~\ref{Fig:WFsim}, where it can be seen that the proposed approach is able to drive the different inputs to the optimum such that the  total power generated by the wind farm is maximized. 
Fig.~\ref{Fig:WFspectrogram} shows the spectrogram, where it can be seen that the amplitude of the power spectrum at the perturbation frequencies become zero as the system converges to the optimum. 
The same simulation results are repeated when the total power measurement is corrupted with a measurement noise of $ w(t) \sim \mathcal{N}(0,2\times10^{-3}) $  MW. The time series as wells as the spectrogram for this case is shown in Fig.~\ref{Fig:WFsim2}. 


\subsection{Example 2: Maximizing heat recovery from a heat exchanger network}
Consider an industrial cluster with $ n $ different  processing plants, the details of which are not important. The different processes in the industrial cluster generates some amount of surplus heat,  which is to be recovered using a heat exchanger network, and transferred to a local district heating system, as shown in Fig.~\ref{Fig:HEx}. 
\begin{figure}
	\centering
	\includegraphics[width=\linewidth]{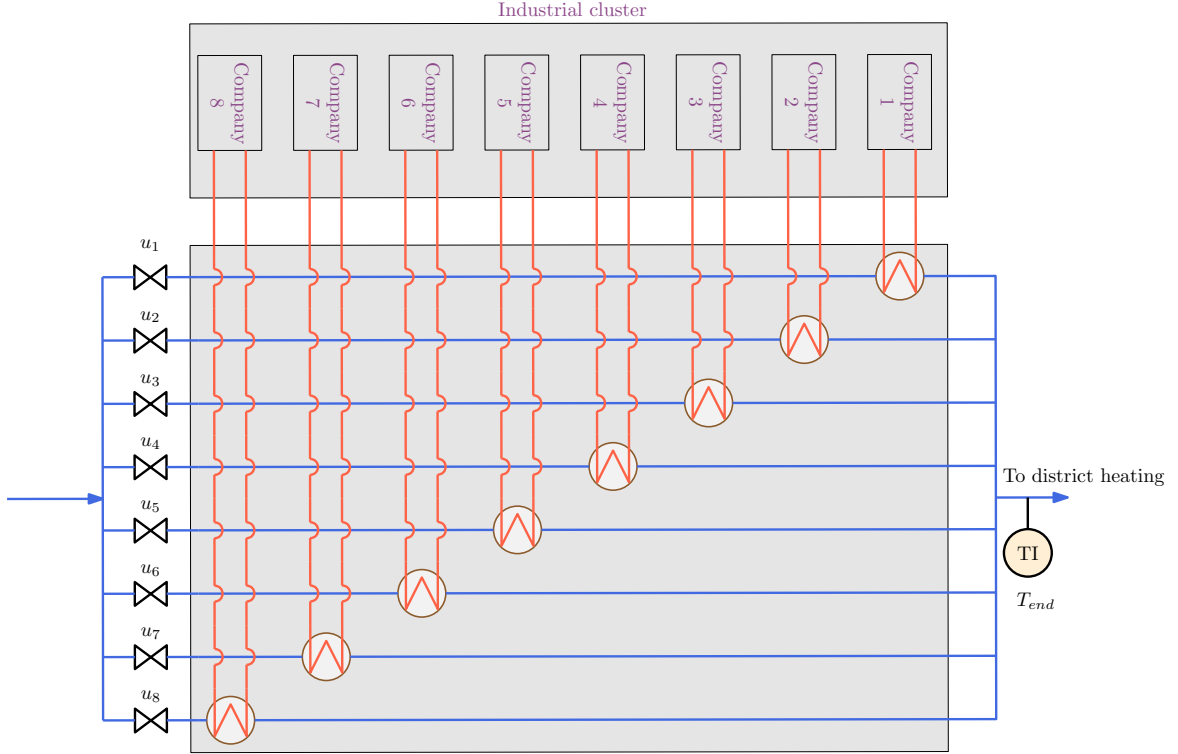}
	\caption{Schematic representation of the heat recovery network from an industrial cluster.}\label{Fig:HEx}
\end{figure}
The cold stream inlet is split into eight different streams, one for each company in the industrial cluster. The objective here is to adjust the split ratio $ u_{i} $ for each stream in order to maximize the total heat recovery, which can be equivalently stated as maximizing the end temperature of the cold stream $ T_{end} $ \cite{jaschke2014optimal}. The optimization problem can then be stated as,
\begin{equation}\label{Eq:Opt}
	\max_{\mathbf{u}} \;T_{end}(\mathbf{u})
\end{equation}
where $\mathbf{u} = [u_{1},\dots,u_{n-1}]$ and $ u_{n} = 1 - \sum_{i=1}^{n-1} u_{i}$. 

As the temperature and flow rate of the hot stream from the industrial cluster varies, or if the overall heat transfer coefficient in the heat exchanger changes (e.g. due to heat exchanger  fouling), this requires adjusting the split ratios in order to maximize the temperature.  This process is inspired by the batch reactor heat recovery challenge problem from \cite{jaschke2012challenge}, where the authors considered a similar problem, and later proposed the concept of controlling the  so-called \textit{J\"{a}schke Temperature}  \cite{jaschke2014optimal} in order to achieve optimal operation. The \textit{J\"{a}schke Temperature}  is essentially the model gradient based on a simplified countercurrent heat exchanger model with arithmetic mean temperature difference.

In this industrial symbiosis network, we assume that no information regarding the heat exchanger network such as temperatures, flow rate, heat transfer coefficient etc. are available, and the model of the system is also not available.  The end temperature $ T_{end} $ is the only measurement that is available for online process optimization. Note that one does not even need to know if the heat exchangers are of countercurrent or concurrent type. The details of the simulator model used in this section can be found in  Appendix~\ref{app:DistrictHeating}. Note that  this model is only used to simulate the process and no knowledge of this model is used by the extremum seeking controller.

We set the initial split ratios to $ u_{i} =0.1$ for $ i=1,\dots,7 $. The temperature of the hot streams are initially $  T_{h,in} = \begin{bmatrix}
	120&130&120&140&120&125&115&110
\end{bmatrix}^{\mathsf{T}} {}^{\circ}$C. After 2000 time steps, the temperature of the first hot stream increases to $ T_{h,in}^{1} = 150^{\circ}$C.
The seven inputs channels are superimposed with sinusoidal signals with frequencies $
	f = \begin{bmatrix}
		6&11&17&23& 31&39&47
	\end{bmatrix}^{\mathsf{T}}/128 $
and amplitude $ a_{i}= $ 0.003 for $ i=1,5,6 $ and $ a_{i}= 0.002 $ for $ i=2,3,4,7 $.
The window length was chosen to be of $ N= $128 samples according to Theorem~\ref{thm:sizeN}. The integral gain was chosen to be $ K_{i} = 7.5\times10^{-6} $ for all $ i = 1,\dots,7 $. The simulation results with the proposed FFT-based multivariable extremum seeking control is show in Fig.~\ref{FFTESC:Fig:DistrictHeating}, where it can be seen that the proposed approach is able to drive the different inputs to the optimum such that the  end temperature  supplied to the district heating  is maximized. 

\begin{figure*}
	\centering
	\includegraphics[width=\linewidth]{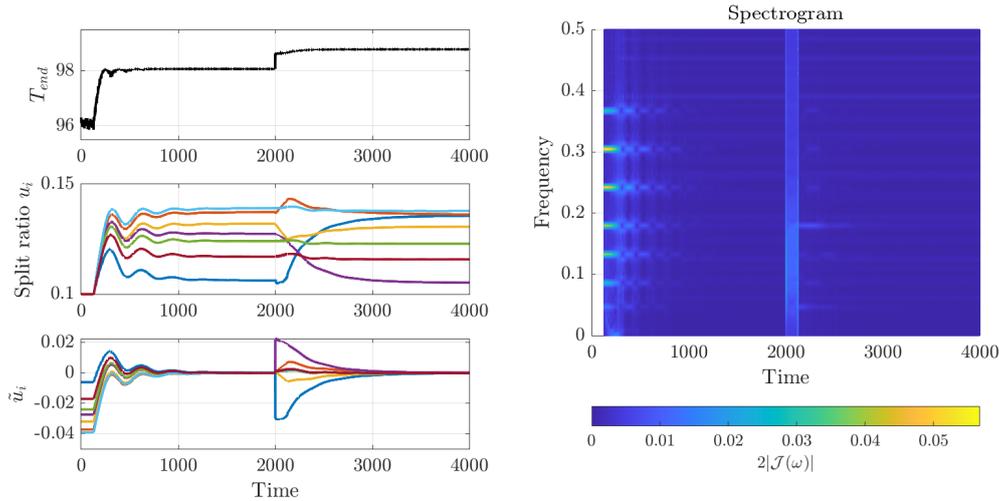}
	\caption{Example 2: Simulation results showing the end temperature recovered from the heat exchanger network, the optimal split ratios for the different flow streams, and the spectrogram showing the power spectrum of the end temeprature measurement as a function of time. }\label{FFTESC:Fig:DistrictHeating}
\end{figure*}

\section{Conclusion}
This paper presented a  multivariable extremum seeking scheme, where the gradient  is estimated using spectral analysis obtained using fast Fourier transform (FFT) as shown in Algorithm~\ref{FFTESC:alg}. The proposed approach requires relatively less tuning parameters compared to the classical extremum seeking control. The choice of the window length $ N $  is formalized in Theorem~\ref{thm:sizeN}. Lemma~\ref{thm:EstErr} showed that the gradient estimation error  due to a moving window length is bounded and Theorem~\ref{thm:convergence} provides bounds on  the gain $ {K}_{i} $ such that the algorithm asymptotically converges to the optimum in the static map setting.   

\section*{Acknowledgments}
Useful discussions with Prof. Sigurd Skogestad at the Norwegian University of Science and Technology, as well as with Dr. P\r{a}l Kittilsen at Equinor Research Center are gratefully acknowledged.

\bibliographystyle{IEEEtran}
\bibliography{DinBib}             
   
   \appendix
   \section{Illustrative  Example}\label{app:example}
Consider the univariate case, where the static map is given by 
\begin{equation}\label{Eq:1DEx}
	J =  -100(u-0.5)^2
\end{equation} 
We use the proposed FFT-based extremum scheme from  Algorithm~\ref{FFTESC:alg}, where the input is perturbed with a frequency of $ f =0.125$~Hz and amplitude of $ a = 0.01 $. We use a window length of $ N=128 $. Fig.~\ref{Fig:GradErrBound} shows the raw and  detrended signal, and the single-sided amplitude spectrum  of the input $ u $ (left hand side subplots) and the cost $ J $ (right hand side subplots) captured at a given time $ t $.  Here we see that the input $ u $ is moving from 0.2 to 0.4, and the single-sided amplitude spectrum has a magnitude of $2|\mathcal{U}(\omega)| = 0.01$ at $ l=0.125 $ that corresponds to the input perturbation as expected. As the  input $ u $ changes from 0.2 to 0.4, the amplitude in the output changes from 0.6 to 0.2 respectively. The tapering of the amplitude is clearly seen in the detrended cost signal in Fig.~\ref{Fig:GradErrBound} (right middle subplot). With a perturbation amplitude of 0.01, this implies that the gradient changes from 60 to 20 in this interval. 
Since spectral analysis shows the power of a signal as the mean squared amplitude at each frequency, the output power spectrum at $ l=0.125 $ indicates that the cost varies with an  amplitude close to 0.4 (i.e. gradient = 40).  Thus, we see that the estimated amplitude in the output power spectrum corresponds to the true amplitude for some input over the last $ N $ samples. This example provides an intuitive understanding of the error bound on the gradient estimate formalized in  Lemma~\ref{thm:EstErr}.
In this simple example, the Hessian $ \mathbf{H}=-200 $ and with $\Delta \mathbf{u} = 0.4-0.2=0.2  $, the upper bound on the gradient estimation error is $ |\epsilon|\leq $ 40, which is reasonable given that the true gradient may be between 60 and 20 as noted. 
\begin{figure*}
	\begin{subfigure}{0.5\linewidth}
	\centering
	\includegraphics[width=\linewidth]{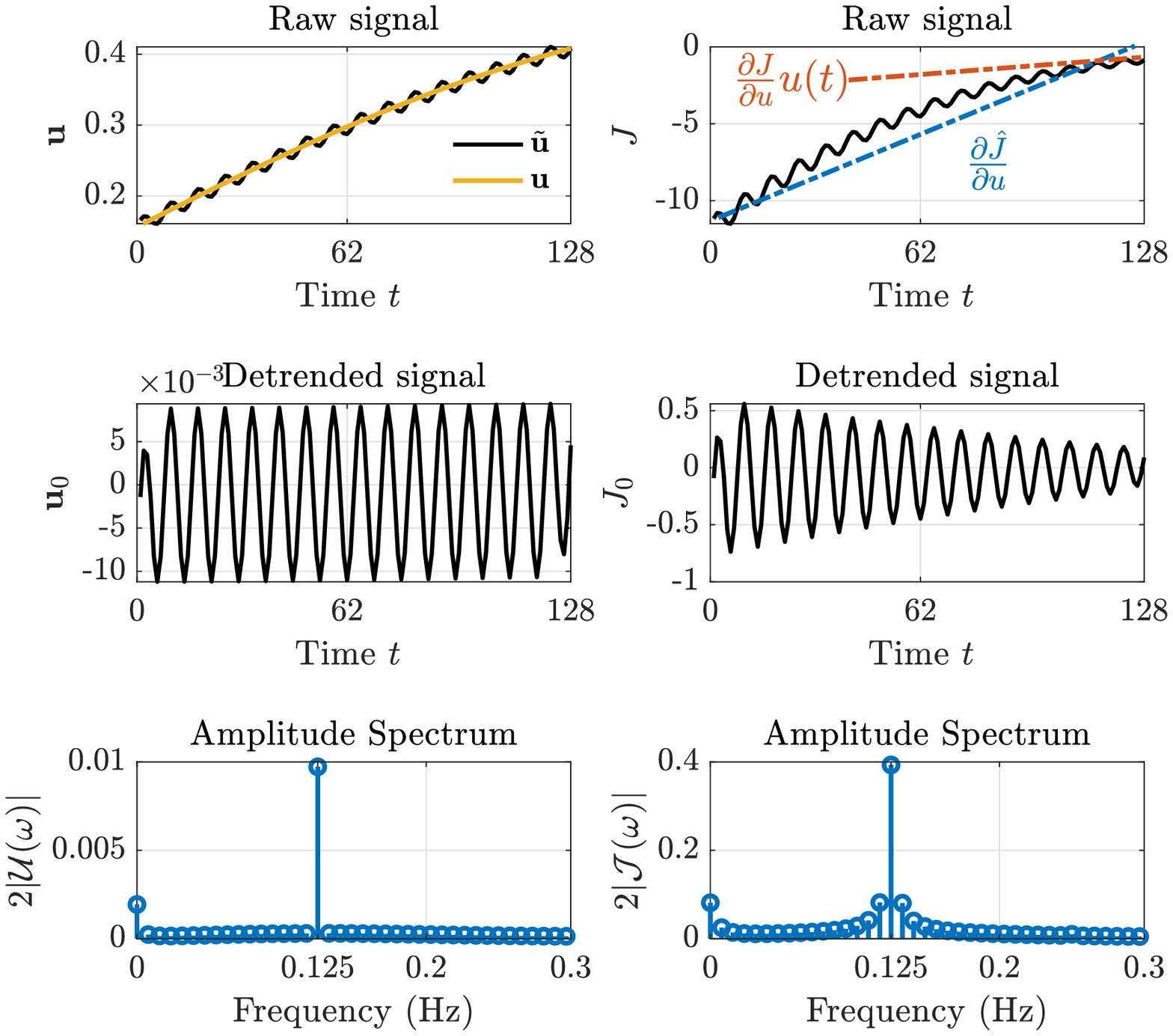}
	\caption{}\label{Fig:GradErrBound}
	\end{subfigure}
\begin{subfigure}{0.5\linewidth}
	\centering
	\includegraphics[width=\linewidth]{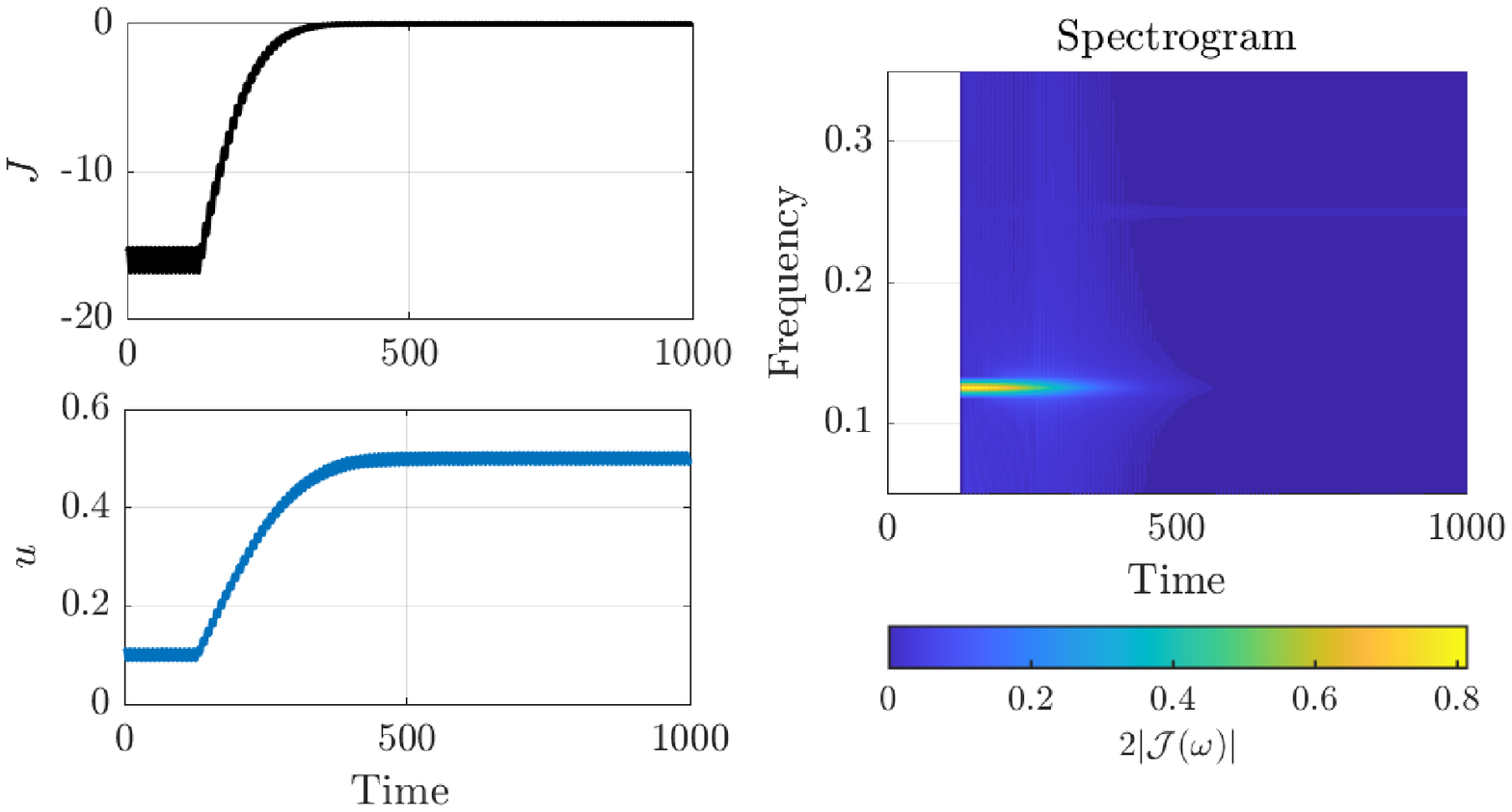}
	\caption{}\label{Fig:Sim1}
\end{subfigure}
	\caption{(a) Illustrative example showing the amplitude error bound when $ \Delta u >0 $. (b) Simulation results using the proposed FFT-based extremum seeking scheme. Right subplot shows the spectrogram. }
\end{figure*}

The simulation results with an integral gain of $ K_{I} = -1.5E-5 $ is shown in Fig.~\ref{Fig:Sim1}. The spectrogram which shows the power spectrum as a function of time is shown in the right subplot in Fig.~\ref{Fig:Sim1}, where it can be seen that the amplitude at $ f = 0.125 $ goes to zero, since at the optimum, the input perturbation leads to a periodic variation of the output at twice the perturbation frequency. This is also seen in the spectrogram at $ f = 0.25 $.

\section{Wind Farm Model}\label{app:windFarm}
In this section we present the wind farm model used in Section~\ref{sec:windfarm} Note that the model is used only to simulate the plant, and no information about the model is used by the extremum seeking algorithm. 

Consider a wind farm with $ n $ wind turbines denoted by the set $ \mathcal{W} = \{1,\dots,n\} $, placed at the co-ordinates $ (x,y) $ relative to a common vertex. The turbines are assumed to be uniform with a diameter $ D $. The axial induction factor for the wind turbines are denoted by $ \mathbf{u} = [u_{1},\dots,u_{n}] $, which are the manipulated variables (degrees of freedom) for optimization. The set of admissible axial induction factors is given by  $ 0\leq u_{i}\leq 0.5 $. Physically, the axial induction factor can be changed by adjusting the blade pitch angle. 

\subsection{Wake interaction model}
When wind passes through a turbine, it creates turbulence, which is carried downwards. This is known as wake effect, which affects the wind turbines that are downwind. To model the wake interaction, we use the momentum balance and the blade element and momentum (BEM) theory \cite{katic1987simple}.
The aggregate wind velocity for any turbine $ i \in \mathcal{W} $ is expressed as,
\begin{equation}\label{Eq:WindVelocity}
V_{i}(\mathbf{u}) = V_{\infty}(1-\delta V_{i}(\mathbf{u}))
\end{equation}
where $ V_{\infty} $ denotes the free stream wind speed (in  m/s), and $ \delta V_{i}(\mathbf{u}) $ denotes the aggregated velocity deficit seen by the turbine $ i $ due to the wake generated by the other upwind turbines. This is given by the expression, 
\begin{equation}\label{Eq:Velocitydeficit}
\delta V_{i}(\mathbf{u}) = 2\sqrt{\sum_{j\in \mathcal{W}:x_{j}<x_{i}} \left(a_{j}\left(\frac{D_{j}}{D_{j+2k(x_{i}-x_{j})}}\right)^2 \frac{A^{overlap}_{j,i}}{A_{i}}\right)^2  }
\end{equation}
where $ A_{i} $ is the sweep area of turbine $ i $ and $ A^{overlap}_{j,i} $ is the overlap of the sweep area between turbine $ i $ and turbine $ j$ at a distance $ x $ downstream.  $ k $ is the roughness co-efficient that defines the angle at which the wake expands out of the turbine. The roughness co-efficient typically takes the value $ k = 0.075 $ for farmlands and $ k = 0.04 $ for offshore wind turbines \cite{sorensen2008adapting}.

\subsection{Power model}
The power (in W) generated by any turbine $ i \in \mathcal{W} $ takes the form, 
\begin{equation}\label{Eq:Power}
P_{i}(\mathbf{u}) = \frac{1}{2}\rho_{air}A_{i}C_{P}(u_{i})V_{i}(\mathbf{u})^3
\end{equation}
where, $ \rho_{air} = 1.225 $ kg/m$ ^3 $ is the air density and $ V_{i}(\mathbf{u}) $  is the wind speed at turbine $ i $ given by  \eqref{Eq:WindVelocity} and \eqref{Eq:Velocitydeficit}. 
$ C_{P} $ is the power efficiency co-efficient characterized by,
\begin{equation}\label{Eq:PowerEff}
C_{P}(u_{i}) = 4u_{i}(1-u_{i})^2
\end{equation}
The total power generated by the wind farm is then simply given by,
\begin{equation}\label{Eq:TotalPower}
P(\mathbf{u}) = \sum_{i\in\mathcal{W}} P_{i}(\mathbf{u})
\end{equation}

\section{Heat exchanger network model}\label{app:DistrictHeating}
Consider a heat exchanger network with $ n $ flow streams with no accumulation as shown in Fig.~\ref{Fig:DistrictHeating}. Cold water at temperature $ T_{c,in} $with a flow rate $ w_{c} $ is split into $ n $ streams, with a split ratio given by $ u_{i} $ for all $ i = 1,\dots,n$. The mass balance dictates that \[ u_{n} = \sum_{i=1}^{n-1}u_{i} \]Each flow stream has an heat exchanger with a hot stream at temperature $ T_{h,in}^{(i)} $ and flow rate $ w_{h}^{(i)} $.  The outlet temperatures of the cold and hot stream from the $ i^{\text{th}} $ heat exchanger is given by,
\begin{align*}
	T_{c,out}^{(i)}&= T_{c,in} + \frac{UA^{(i)} \Delta T_{lm}^{(i)}}{u_{i}w_{c} c_{p}}\\
	T_{h,out}^{(i)} &= T_{h,in}^{(i)} - \frac{UA^{(i)} \Delta T_{lm}^{(i)}}{w_{h}^{(i)} c_{p}}
\end{align*}
where $ UA^{(i)} $ is the overall heat transfer coefficient of heat exchanger $ i $ (in W/$^\circ$C), $ c_{p} $ is the specific heat capacity in kJ/kg/$^\circ$C,
\begin{align*}
\Delta T_{lm}^{(i)}&:= \frac{\Delta T_{1}^{(i)} - \Delta T_{2}^{(i)}  }{\ln\left(\Delta T_{1}^{(i)}/\Delta T_{2}^{(i)} \right)} \\
&\approx \left[ \Delta T_{1}^{(i)}\Delta T_{2}^{(i)} \left(\frac{\Delta T_{1}^{(i)}+\Delta T_{2}^{(i)}}{2}\right) \right]^{1/3}
\end{align*}
is the log-mean temperature difference, and for counter-current flow $ \Delta T_{2}^{(i)}:=T_{h,out}^{(i)} - T_{c,in}^{(i)} $ and $ \Delta T_{1}^{(i)}:= T_{h,in}^{(i)} - T_{c,out}^{(i)}$.
Assuming linear mixing, the end tmperature $ T_{end} $ is given by
\begin{equation}
	T_{end} = \sum_{i=1}^{n}u_{i}T_{c,out}^{(i)}
\end{equation}


\end{document}